\pdfoutput=1
\documentclass[journal,onecolumn]{IEEEtran}

\usepackage{longtable}
\usepackage{makecell}
\usepackage{mdwlist}%enumerate
\usepackage{amssymb} %mathbb
\usepackage{amsthm}
\usepackage{array}
\usepackage{amsmath}
\usepackage{enumerate}
\usepackage{longtable}
\usepackage{graphicx}
\usepackage{galois}
\usepackage[hidelinks,linkcolor=black,bookmarks,bookmarksopen,bookmarksdepth=2]{hyperref}
\usepackage[justification=centering]{caption}
\usepackage[numbers, sort&compress]{natbib}

\makeatletter

\newcommand{\Z}{\mathbb{Z}}

\newcommand{\F}{\mathbb{F}}
\newcommand{\N}{\mathbb{N}}

\DeclareMathOperator{\Ima}{Im}
\DeclareMathOperator{\lcm}{lcm}

\makeatother

\begin{document}
\renewcommand\thesection{\arabic{section}}
\renewcommand\thetable{\arabic{table}}

\theoremstyle{plain}
\newtheorem{thms}{Theorem}[section]
\newtheorem{prop}[thms]{Proposition}
\newtheorem{lemma}[thms]{Lemma}
\newtheorem{fact}[thms]{Fact}
\newtheorem{conc}[thms]{Corollary}
\newtheorem{conjecture}[thms]{Conjecture}

\theoremstyle{definition}
\newtheorem{definition}{Definition}
\newtheorem{example}[thms]{Example}

\theoremstyle{remark}
\newtheorem{remark}[thms]{Remark}

\title{Some Zero-Difference Functions Over $\mathbb{Z}_n$ Using Cyclotomies}
\author{Zongxiang~Yi
\thanks{Z. Yi is with the College of Mathematics and Information Science, Guangzhou University, Guangzhou, Guangdong 510006, P.R. China e-mail:tpu01yzx@gmail.com}% <-this % stops a space
}
\maketitle

%\subjclass{Primary 06E30, 05B10, 94B25}

\begin{abstract}
A generic method to construct zero-difference functions (ZDFs) on algebraic rings is proposed in this paper. Then this method is used over some rings $\Z_{p^k}$, where $p$ is a prime number and $k\ge 2$ is a positive integer, and for some other special rings.
\end{abstract}
\IEEEpeerreviewmaketitle

\section{Introduction}
Let $(A,+)$ and $(B, +)$ be two finite Abelian groups. A function from $A$ to $B$ is an $(n,m,\lambda)$ zero-difference balanced function (ZDBF), if there exists a constant number $\lambda$ such that for any nonzero $a \in A$,
$$|\{x \in A \mid f(x+a)-f(x)=0 \}|=\lambda,$$
where $n=|A|$ and $m=|\Ima(f)|$. \citeauthor{carlet2004highly} first proposed the concept of ZDBF in \citeyear{carlet2004highly} \cite{carlet2004highly}. Some optimal objects, such as constant composition codes (CCC), constant weight codes (CWC), difference systems of sets (DSS) and frequency-hopping sequences (FHS), can be obtained by ZDBFs. Therefore, many researchers have been working on constructing ZDBFs (see \cite{ding2009optimal, ding2008optimal, carlet2004highly, zhou2012some, wang2014sets, cai2013new, ding2014three, zha2015cyclotomic, zhifan2015zero, cai2017zero, li2017generic, yi2018generic} and the references therein). However \citeauthor{buratti2019partitioned} recently pointed out that many results on ZDBF were known as partitioned difference family (PDF) \cite{buratti2019partitioned} and proposed the "philosophy" on block sizes. It makes us turn to the problem of generalizing the concept of ZDBF.

In \citeyear{carlet2014quadratic}, \citeauthor{carlet2014quadratic} proposed another concept called differentially $\delta$-vanishing \cite{carlet2014quadratic}. A function from $A$ to $B$ is differentially $\delta$-vanishing, if for any nonzero $a \in A$,
$$1\le |\{x \in A \mid f(x+a)-f(x)=0 \}| \le \delta,$$
where $n=|A|$ and $m=|\Ima(f)|$. Any $(n,m,\lambda)$ ZDBF is differentially $\lambda$-vanishing. But there had been little research on this concept, until \citeauthor{jiang2016generalized} proposed a related concept called generalized zero-difference balanced function in \citeyear{jiang2016generalized} \cite{jiang2016generalized}.

A function from $A$ to $B$ is an $(n, m, S)$ generalized zero-difference balanced function (G-ZDBF), if there exists a constant set $S \subset \N$ such that for any nonzero $a \in A$,
$$|\{x \in A \mid f(x+a)-f(x)=0 \}|\in S,$$
where $n=|A|$ and $m=|\Ima(f)|$. Then some objects can be obtained by ZDFs \cite{jiang2016generalized, Jiang2016New, liu2016some}, but they are not optimal. The reason is that the size of $S$ is too large, which implies they are not really zero-difference balanced. So we refer to use zero-difference function (ZDF) instead of generalized zero-difference balanced function (G-ZDBF). A ZDF is called trivial if it is a ZDBF. Moreover, the set $S$ should have a small size and small differences between two distinct elements. We call these two requirements, the "philosophy" on zero-difference. Other requirements should be impose on $S$ depending on different applications. For example, to make the FHSs meets some theoretic bound, the average of the elements in $S$ should close to some constant number $C(n,m)$.

In this paper we concern those $(n,m,S)$ ZDFs with $|S|=2$ or $|S|=3$. The contributions of this paper are as follows:
\begin{enumerate}[(1)]
\item We propose a generic method to construct ZDFs on algebraic rings;
\item We construct several classes of $(n,m,S)$ ZDFs on residual class rings $\Z_n$ with $S=2$ or $S=3$;
\end{enumerate}

The rest of this paper is organized as follows. In Section \ref{se:ogr}, we show the method to construct ZDFs on algebraic rings. In Section \ref{se:orcr}, we construct ZDFs on residual class rings. In Section \ref{se:con}, we make conclusions.

\section{Construction On Algebraic Ring}\label{se:ogr}

%\subsection{The Construction of ZDFs}
In this paper, we assume that $(R, +, \times)$ is a ring with identity. % and for any $x, y \in R$, $xy = x \times y$.
Firstly, we introduce a lemma about partition formed by a subgroup.
\begin{lemma}\cite[pp. 8-10]{eilenberg1974automata}\label{lm:cosets_partition}
Let $(R, +, \times)$ be a ring of order $n$, and let $G$ be a subgroup of $(R, \times)$. Then
$$D_G=\{rG \mid r \in R\}$$
is a partition of $R$, where $rG=\{ rg \mid g \in G\}$ is called the coset of $G$.
\end{lemma}

Given a subgroup $G$ of $(R, \times)$, a partition $D_G$ of $R$ is obtained. Now we define a function $g_G$ from $R$ to $D_G$: $g_G(x)=rG$ where $rG$ contains $x$, and there exists a bijective function $h_G$ from $D_G$ to $\Z_{|D_G|}$. So by function composition, a function $f_G$ from $R$ to $\Z_{|D_G|}$ is obtained, i.e.,
$$f_G(x)=h_G(g_G(x)).$$
%The partition $D_G$ is called the partition induced by $G$.
The function $f_G(x)$ is called the coset index function induced by $G$.
%Function $g_G(x)$ is called the coset generating function which returns the unique coset $rG$ containing $x$. Function $h_G(x)$ is called the numbering function which numbers every coset in the partition from 0 with step 1. Function $f_G(x)$ is called the coset index function which returns the number of coset that $x$ belongs to.
\begin{remark}Lemma \ref{lm:cosets_partition} guarantees that $g_G$ is well-defined.
\end{remark}
Let $(R, + ,\times)$ be a ring, and let $x$ be an unknown in $R$. Here are some notations:
\begin{enumerate}[(1)]
% \item $S(g, a)$ denotes the solution set of equation $xg=a$ and $N(g,a)$ the number of solutions of that equation;
 \item $S(G, a)$ denotes the union of the solution sets of equations $x(g-1)=a$ for every $g \in G$ and $N(G,a)$ denotes the size of $S(G, a)$;
 \item $d(G)$ denotes the set of possible sizes of all cosets $rG$ when $r$ runs over $R$, and $M(G,a)$ denotes the number of elements $r \in R$ such that the size of the coset $rG$ is $a$;
 \item $R^\times$ denotes the set of all invertible elements in $(R, \times)$ and $R^*$ denotes the set of all nonzero elements in $R$.
\end{enumerate}

Now we propose our method to construct ZDFs on a ring $(R, +, \times)$.
\begin{thms}\label{th:construct_gzdb_on_generic_ring}
Let $(R, +, \times)$ be a ring of order $n$, and let $G$ be a subgroup of $(R, \times)$. $f_G(x)$ denotes the coset index function induced by $G$. Then $f_G(x)$ is an $(n, m, S)$ ZDF from $(R,+)$ to $\Z_{m}$, where
$m=\sum_{a \in d(G)}\frac{M(G,a)}{a}$, $S=\{ N(G, a) \mid a \in R\backslash\{0\}\}$.
\end{thms}
\begin{proof}
Consider the three parameters of ZDF. The first parameter is obvious. The second parameter can be obtained by Lemma \ref{lm:cosets_partition}. The following is to show that the third parameter is correct. Since $h_G$ is bijective, we have
$$\{x \in R \mid f_G(x+a)=f_G(x) \} = \{x \in R \mid g_G(x+a)=g_G(x) \}.$$
Then it suffices to show that for $\forall a \in R\backslash\{0\}$,
$$\{x \in R \mid g_G(x+a)=g_G(x) \}=\bigcup_{g \in G }\{x \in R \mid x(g-1)=a \}.$$

On one hand, for $\forall x \in \{x \in R \mid g_G(x+a)=g_G(x) \}$, there exists an element $r \in R$ and two elements $g_1, g_2 \in G$ such that
\begin{equation*}\begin{cases}
x+a = rg_1 \\
x =rg_2
\end{cases}.\end{equation*}
It implies
$$x+a=rg_1=rg_2g_2^{-1}g_1=xg_2^{-1}g_1.$$
Then we have $x(g_2^{-1}g_1 - 1) =a$. Therefore $x \in \bigcup_{g \in G }\{x \in R \mid x(g-1)=a \}$.

On the other hand, for $\forall x \in \bigcup_{g \in G }\{x \in R \mid x(g-1)=a \}$, there exists an element $g \in G$ such that
$$x(g-1)=a.$$
It implies
$$x+a=xg\in xG.$$
Then we have $g_G(x+a)=g_G(x)=xG$. Therefore $x \in \{x \in R \mid g_G(x+a)=g_G(x) \}$.

Finally, we have
$$\{x \in R \mid g_G(x+a)=g_G(x) \}=\bigcup_{g \in G }\{x \in R \mid x(g-1)=a \}.$$
\end{proof}

Note that $1\in G$ and $a \ne 0$. So when $g=1$, $\{x \in R \mid x(g-1)=a \}=\emptyset$. If the group $G$ happens to satisfy the condition
$$(G-1)\backslash\{0\} \subset R^\times, $$
then the function constructed by Theorem \ref{th:construct_gzdb_on_generic_ring} is also a ZDBF. Corollary \ref{con:construct_zdb_on_generic_ring} is almost the same as Theorem 1 in \cite{yi2018generic}. However Corollary \ref{con:construct_zdb_on_generic_ring} does not require that $R$ be commutative.

Let $R$ be residual class ring $\Z_n$ or the product of finite fields $\F_q$, then the ZDBFs in \cite{cai2013new, zha2015cyclotomic, ding2014three} can be retrieved by our method. So Corollary \ref{con:construct_zdb_on_generic_ring} can be viewed as a generalization of those results.
\begin{conc}\label{con:construct_zdb_on_generic_ring}
Let $(R, +, \times)$ be a ring of order $n$, and let $G$ be a subgroup of $(R, \times)$. $f_G(x)$ is the coset index function. If $G$ satisfies the condition
$$(G-1)\backslash\{0\} \subset R^\times, $$
then $f_G(x)$ is an $(n, \frac{n-1}{k}+1, k-1)$ ZDBF from $(R,+)$ to $\Z_{m}$, where $m=\frac{n-1}{k}+1$, $k=|G|$, and $G-1=\{a-1 \mid a \in G\}$
\end{conc}
%\begin{proof}
%Note that $\forall \alpha \in R$,
%\begin{equation*}
%|\alpha G| = \begin{cases}
%m, & \alpha \ne 0 \\
%1, & \alpha = 0.
%\end{cases}
%\end{equation*}
%\and
%\begin{equation*}\begin{split}
%&\bigcup_{g \in G }\{x \in R \mid (g-1)x=a \} \\
%=&(\bigcup_{g \in G\backslash\{1\} }\{(g-1)^{-1}a \})\bigcup \emptyset \\
%=&\bigcup_{g \in G\backslash\{1\} }\{(g-1)^{-1}a \}
%\end{split}\end{equation*}
%
%$$\bigcup_{g \in G }\{x \in R \mid (g-1)x=a \}$$
%
%
%\end{proof}

\section{On Residual Class Ring}\label{se:orcr}
In this section, we consider applying Theorem \ref{th:construct_gzdb_on_generic_ring} to the case that $R=\Z_n$ and $G=\langle e \rangle$. Since the case that $n$ is prime, i.e., $\Z_n$ is a finite field,  is already considered in \cite{ding2014three}, the cases that $n$ is prime power will be discussed in this section.

We remark that by Theorem \ref{th:construct_gzdb_on_generic_ring} the problem of constructing ZDFs comes down to the problem of solving linear equations. So we give two lemmas about solving linear equations in $\Z_n$ before our constructions.
\begin{remark}
Note that $Z_n$ is communicative, so $ax=b$ and $xa=b$ are the same. But we just used to write the linear equation as $ax=b$.
\end{remark}
\begin{lemma} \cite{ireland1990classical}\label{lm:about_solutions_on_zn}
The congruence $ax\equiv b \pmod{n}$ has solutions if and only if $d=\gcd(a,n) \mid b$. If $d\mid b$, then there are exactly $d$ solutions. If $x_0$ is a solution, then the other solutions are given by $x_0+n', x_0+2n', \ldots, x_0+(d-1)n'$, where $n'=\frac{n}{d}$.
\end{lemma}

\begin{lemma}\label{lm:about_solution_on_primepower}
Let $p$ be a prime number, and let $k$ be a positive integer. Suppose $i, a \in \Z_{p^k}\backslash\{0\}$. If $\gcd(i,p)=1$ and $p\mid a$, then for any integer $t$ , the solution of the congruence $(tp^{k-1}-i)x\equiv a \pmod{p^k}$ is independent of $t$ .
\end{lemma}
\begin{proof}
Since $\gcd(i,p)=1$, it follow from Lemma \ref{lm:about_solutions_on_zn} that $-ix\equiv a \pmod{p^k}$ has exactly one solution denoted by $x_0$. Note that $p\mid a$, so it must have $p \mid x_0$. Thus we have
\begin{equation*}\begin{split}
&(tp^{k-1}-i)x_0 \pmod{p^k} \\
\equiv &tp^{k-1}x_0-ix_0 \pmod{p^k} \\
\equiv & -ix_0 \pmod{p^k} \\
\equiv & a \pmod{p^k}
\end{split}\end{equation*}
Therefore, $x_0$ is also a solution of congruence $(tp^{k-1}-i)x\equiv a \pmod{p^k}$. Since $\gcd(tp^{k-1}-i, p^k)=1$, the congruence $(tp^{k-1}-i)x\equiv a \pmod{p^k}$ has only one solution $x_0$. Obviously $x_0$ is independent of $t$.
\end{proof}

Next we will show some classes of ZDFs by choosing different prime powers $p^k$ , where $p$ is a prime number and $k \ge 2$ is an integer.
%Let the generators $e$ of $G$ be $p^{k-1}-1$.

\subsection{Case $n=4$}
Applying Theorem \ref{th:construct_gzdb_on_generic_ring} on $\Z_4$, the situations are clear since there are only two subgroups of $Z_4^\times$. All the ZDFs are listed in Table \ref{tb:gzdb_on_z4}
\begin{table}[!htbp]
\centering
\caption{\text{$(n,m,S)$ ZDFs on $\Z_4$}}\label{tb:gzdb_on_z4}
\begin{tabular*}{\linewidth }{c|c|c|c}
\hline
$G$ & $D_G$ & $g_G$ & $(n,m,S)$   \\
\hline
$\{1\}$ &$\{\{0\}, \{1\}, \{2\}, \{3\} \}$ & $g_G(x)=\{x\}$ & $(4,4,\{0\})$  \\
\hline
$\{1, 3\}$ & $\{\{0\}, \{1, 3\}, \{2\} \}$ & $g_G(x)=\begin{cases}
\{0\}, & x=0, \\
\{2\}, & x=2, \\
\{1,3\}, & x \in \{1,3\}.
\end{cases}$ & $(4,3,\{0, 2\})$  \\
\hline
\end{tabular*}
\end{table}

\subsection{Case $n=2^k$}
Let $n=2^k$, $G=\langle 2^{k-1}-1 \rangle$, where $k > 2$. Then we have Theorem \ref{th:construct_gzdb_on_2k} shown as follows.
\begin{thms}\label{th:construct_gzdb_on_2k}
Let $n=2^k$, where $k > 2$, and let $G=\langle 2^{k-1}-1 \rangle$ be a subgroup of $Z_n^\times$. Then the coset index function $f_G$ is a $(2^k, 2^{k-1}+1, \{0, 2\})$ ZDF.
\end{thms}
\begin{proof}
The proof is consisted by four steps:
\begin{enumerate}[(1)]
\item We claim that $G= \{1, 2^{k-1}-1\}$ and $|G|=2$, since
\begin{equation*}\begin{split}
&(2^{k-1}-1)^2 \pmod{2^k} \\
\equiv &2^{2(k-1)}-2\cdot2^{k-1}+1\pmod{2^k} \\
\equiv &1\pmod{2^k},
\end{split}\end{equation*}
and
$$2^{k-1}-1 \not\equiv 1\pmod{2^k}.$$

\item We assert that for $\forall \alpha \in \Z_{2^k}$, it has
\begin{equation*}
|\alpha G| = \begin{cases}
1, & 2^{k-1} \mid \alpha, \\
2, & 2^{k-1} \nmid \alpha.
\end{cases}
\end{equation*}
If $\alpha \equiv \alpha(2^{k-1}-1) \pmod{2^k}$, then it implies $\alpha\cdot 2\cdot(2^{k-2}-1) \equiv 0 \pmod{2^k}$. We have $\gcd(2^{k-2}-1, 2^k)=1$, since $k>2$. So it must have $2^{k-1} \mid \alpha$. \\
\item The size of image of $f_G$ can be obtained as follows.
\begin{equation*}\begin{split}
&|\Ima(f_G)|=|D_G| \\
=&\frac{2^k-2}{2}+2 \\
=&2^{k-1}+1.
\end{split}\end{equation*}

\item Based on Theorem \ref{th:construct_gzdb_on_generic_ring}, it suffices to show that for $\forall \alpha \in \Z_{2^k}\backslash\{0\}$, it has
$$|\bigcup_{g \in G }\{x \in \Z_{2^k} \mid (g-1)x\equiv \alpha \pmod{2^k} \}| \in \{0, 2\}.$$
Obviously $\bigcup_{g=1}\{x \in \Z_{2^k} \mid (g-1)x\equiv \alpha \pmod{2^k} \} = \emptyset$, hence we only have to show that
$$|\{x \in \Z_{2^k} \mid (2^{k-1}-2)x\equiv \alpha \pmod{2^k} \}| \in \{0, 2\}.$$
Note that $\gcd(2^{k-1}-2, 2^k)=2$. By Lemma \ref{lm:about_solutions_on_zn}, the congruence $(2^{k-1}-2)x\equiv \alpha \pmod{2^k}$ have solutions if and only if $2\mid \alpha$, and if $2\mid \alpha$, there are exactly $2$ solutions. So
\begin{equation*}
N(\langle 2^{k-1}-1 \rangle,\alpha) = \begin{cases}
2, & \text{$2\mid \alpha$}, \\
0, & 2\nmid\alpha.
%2^k & \alpha = 0 \\
\end{cases}.
\end{equation*}

\end{enumerate}
Finally, $f_G$ is a $(2^k, 2^{k-1}+1, \{0, 2\})$ ZDF.
\end{proof}

\subsection{Case $n=p^2$}
Let $n=p^2$, $G=\langle p-1 \rangle$, where $p$ is an odd prime. We have Theorem \ref{th:construct_gzdb_on_prime2}.
\begin{thms}\label{th:construct_gzdb_on_prime2}
Let $n=p^2$, where $p$ is an odd prime, and let $G=\langle p-1 \rangle$ be a subgroup of $\Z_n^\times$. Then the coset index function $f_G$ is a $(p^2, p, \{p, p^2-p+1\})$ ZDF.
\end{thms}
\begin{proof}
The proof is consisted by four steps:
\begin{enumerate}[(1)]
\item We claim that $G= \{ (-1)^{t-1}(tp-1)\pmod{p^2} \mid t=0,1,\ldots,(2p-1)\}$, and $|G|=2p$. We have
\begin{equation*}\begin{split}
& (p-1)^{2p}\equiv 1 \pmod{p^2}, \\
& (p-1)^{2}\equiv -2p+1\not\equiv 1 \pmod{p^2}, \\
 \mbox{and}\ & (p-1)^{p}\equiv -1\not\equiv 1 \pmod{p^2}.
\end{split}\end{equation*}
Hence the multiplicative order of $p-1$ is $2p$, i.e., $|G|=2p$. Furthermore,
\begin{equation*}\begin{split}
G &= \{ (p-1)^t \pmod{p^2} \mid t=0,1,\ldots,(2p-1)\} \\
&=\{ (-1)^{t-1}(tp-1) \pmod{p^2} \mid t=0,1,\ldots,(2p-1)\}
\end{split}\end{equation*}
%On the other hand, for $\forall t \in \N$, we have
%\begin{equation*}\begin{split}
%&(-1)^{t-1}(tp-1)\cdot (p-1) \pmod{p^2}\\
%=&(-1)^{t-1}(tp^2-(t+1)p+1) \pmod{p^2} \\
%\equiv &(-1)^{t-1}(-(t+1)p+1) \pmod{p^2} \\
%\equiv & (-1)^{t}((t+1)p-1) \pmod{p^2},
%\end{split}\end{equation*}
%and $(-1)^{0-1}(0\cdot p-1)=1$. Therefore, $G= \{ (-1)^{t-1}(tp-1) \mid t=0,1,\ldots,(2p-1)\}$.

\item We assert that for $\forall \alpha \in \Z_{p^2}$, it has
\begin{equation*}
|\alpha G| = \begin{cases}
1, & \alpha = 0, \\
2, & p \mid \alpha \text{ and } \alpha \ne 0, \\
2p, & p \nmid \alpha.
\end{cases}
\end{equation*}

If $\alpha=0$, then obviously $|\alpha G|=|\{ 0\}| = 1$.

If $p \nmid \alpha$, then $\alpha$ has a multiplicative inverse. As a result, $|\alpha G|=|G|=2p$.

If $p \mid \alpha$ and $\alpha \ne 0$, then it must have
$$\alpha G =\{\alpha, \alpha(p-1) \pmod{p^2}\},$$
since
\begin{equation*}\begin{split}
&\alpha (p-1)^2 \pmod{p^2} \\
\equiv& \alpha p^2-2p\alpha+\alpha \pmod{p^2} \\
\equiv & \alpha \pmod{p^2}.
\end{split}\end{equation*}
Moreover, if $\alpha \equiv \alpha(p-1) \pmod{p^2}$, i.e., $\alpha (p-2) \equiv 0 \pmod{p^2}$, then $\alpha =0$ which is a contradiction.

\item The size of image of $f_G$ can be obtained as follows.
\begin{equation*}\begin{split}
&|\Ima(f_G)|=|D_G| \\
=& \frac{p(p-1)}{2p}+\frac{p-1}{2}+1 \\
=& p.
\end{split}\end{equation*}

\item We make a partition of the group $G$ as $\{G_1, G_2\}$, where
\begin{equation*}\begin{split}
G_1 =& \{ (-1)^{t-1}(tp-1) \pmod{p^2}\mid t=0,2,4,\ldots,(2p-2)\} \\
=& \{ -tp+1 \mid t=0,2,4,\ldots,(2p-2)\},\\
G_2 = & \{ (-1)^{t-1}(tp-1) \pmod{p^2}\mid t=1,3,5,\ldots,(2p-1)\} \\
=& \{ tp-1 \mid t=1,3,5,\ldots,(2p-1)\}.
\end{split}\end{equation*}
%Obviously
%$$G=G_1 \bigcup G_2, G_1 \bigcap G_2 = \emptyset.$$
%That is to say $\{G_1, G_2\}$ is indeed a partition of $G$.
Due to Theorem \ref{th:construct_gzdb_on_generic_ring}, it suffices to show that for $\forall \alpha \in \Z_{p^2}\backslash\{0\}$,
$$|\bigcup_{g \in G }\{x \in \Z_{p^2} \mid (g-1)x\equiv \alpha \pmod{p^2} \}| \in \{p, p^2-p+1\}.$$

For $G_1$,
\begin{equation*}\begin{split}
&S(G_1,\alpha) \\
=&\bigcup_{g \in G_1 }\{x \in \Z_{p^2} \mid (g-1)x\equiv \alpha \pmod{p^2} \} \\
=&\bigcup_{i=0}^{p-1}\{x \in \Z_{p^2} \mid ((-2ip+1)-1)x\equiv \alpha \pmod{p^2} \} \\
=&\bigcup_{i=0}^{p-1}\{x \in \Z_{p^2} \mid -2ipx\equiv \alpha \pmod{p^2} \}.
\end{split}\end{equation*}
Note that $\alpha \ne 0$. Hence $\{x \in \Z_{p^2} \mid 0x\equiv \alpha \pmod{p^2}\} = \emptyset$.
It follows from Lemma \ref{lm:about_solutions_on_zn} that
\begin{equation*}\begin{split}
&S(G_1,\alpha) \\
=&\bigcup_{i=1}^{p-1}\{x \in \Z_{p^2} \mid -2ipx\equiv \alpha \pmod{p^2} \} \\
=&\begin{cases}
\bigcup_{i=1}^{p-1}\bigcup_{k=0}^{p-1}\{[\frac{\alpha/p}{-2i}]_p+kp \}, & p \mid \alpha, \\
\emptyset, & p\nmid \alpha,
\end{cases}\\
=&\begin{cases}
\Z_{p^2}^\times, & p \mid \alpha, \\
\emptyset, & p\nmid \alpha,
\end{cases}
\end{split}\end{equation*}
where $[\frac{b}{a}]_n$ denotes a solution of the congruence $ax\equiv b \pmod{n}$. So we have
\begin{equation*}\begin{split}
N(G_1, \alpha) =&|S(G_1, \alpha)| \\
=&\begin{cases}
p^2-p, & p \mid \alpha, \\
0, & p\nmid \alpha.
\end{cases}
\end{split}\end{equation*}

For $G_2$,
\begin{equation*}\begin{split}
&S(G_2,\alpha) \\
=&\bigcup_{g \in G_2 }\{x \in \Z_{p^2} \mid (g-1)x\equiv \alpha \pmod{p^2} \} \\
=&\bigcup_{i=0}^{p-1}\{x \in \Z_{p^2} \mid (((2i+1)p-1)-1)x\equiv \alpha \pmod{p^2} \} \\
=&\bigcup_{i=0}^{p-1}\{x \in \Z_{p^2} \mid ((2i+1)p - 2)x\equiv \alpha \pmod{p^2} \} \\
=&\bigcup_{i=0}^{p-1}\{[\frac{\alpha}{(2i+1)p - 2}]_{p^2} \}. \\
\end{split}\end{equation*}

Due to Lemma \ref{lm:about_solution_on_primepower}, we have
\begin{equation*}\begin{split}
N(G_2, \alpha) =&|S(G_2, \alpha)| \\
=&\begin{cases}
1, & p \mid \alpha, \\
p, & p \nmid \alpha.
\end{cases}
\end{split}\end{equation*}

Moreover, for $\forall x \in \bigcup_{i=0}^{p-1}\{[\frac{\alpha}{(2i+1)p - 2}]_{p^2} \}$, it must have $p \mid x$, by Lemma \ref{lm:about_solution_on_primepower}. So $x \notin \Z_{p^2}^\times$.

Finally, for $\forall \alpha \in \Z_{p^2}\backslash\{0\}$, it has
\begin{equation*}\begin{split}
&S(G_1,\alpha) \bigcap S(G_2,\alpha) \\
=&\begin{cases}
\Z_{p^2}^\times \bigcap (\bigcup_{i=1}^{p-1}\{[\frac{\alpha}{(2i+1)p - 2}]_{p^2} \})& p \mid \alpha , \\
\emptyset \bigcap (\bigcup_{i=1}^{p-1}\{[\frac{\alpha}{(2i+1)p - 2}]_{p^2} \}), & p\nmid \alpha,
\end{cases}\\
=&\begin{cases}
\emptyset, & p \mid \alpha, \\
\emptyset, & p \nmid \alpha,
\end{cases}\\
=&\emptyset.
\end{split}\end{equation*}
Consequently,
\begin{equation*}\begin{split}
N(G, \alpha)=&N(G_1, \alpha) + N(G_2, \alpha)\\
=&\begin{cases}
p^2-p+1, & p \mid \alpha, \\
p, & p \nmid \alpha.
\end{cases}
\end{split}\end{equation*}

\end{enumerate}
Finally, $f_G$ is a $(p^2, p, \{p, p^2-p+1\})$ ZDF.
\end{proof}

\subsection{Case $n=p^k$}
Let $n=p^k$, $G=\langle p^{k-1}-1 \rangle$ where $p$ is an odd prime number and $k > 2$. We have the Theorem \ref{th:construct_gzdb_on_prime2} as follow.
\begin{thms}\label{th:construct_gzdb_on_primek}
Let $n=p^k$, where $p$ is an odd prime number. Let $G=\langle p^{k-1}-1  \rangle$ be a subgroup of $\Z_n^\times$. Then the coset index function $f_G$ is a $(p^k, \frac{2p^{k-1}-p^{k-2}+1}{2}, \{1, p, p^k-p^{k-1}+1\})$ ZDF.
\end{thms}
\begin{proof}The proof is similar with that of Theorem \ref{th:construct_gzdb_on_prime2}. We have
\begin{enumerate}[(1)]
\item $G= \{ (-1)^{t-1}(tp^{k-1}-1) \pmod{p^k}\mid t=0,1,\ldots,(2p-1)\}$ and $|G|=2p$;
\item for $\forall \alpha \in \Z_{p^k}$, it has
\begin{equation*}
|\alpha G| = \begin{cases}
1, & \alpha = 0, \\
2, & p \mid \alpha \text{ and } \alpha \ne 0, \\
2p, & p \nmid \alpha;
\end{cases}
\end{equation*}
\item $|\Ima(f_G)|=p^{k-1}-\frac{p^{k-2}-1}{2}$;
\item for $\forall \alpha \in \Z_{p^k}\backslash\{0\}$, it has
\begin{equation*}\begin{split}
N(G, \alpha)=&\begin{cases}
p^{k}-p^{k-1}+1, & p^{k-1} \mid \alpha, \\
p, & p \nmid \alpha, \\
1, & p^{k-1}\nmid \alpha \text{ and } p \mid \alpha.
\end{cases}
\end{split}\end{equation*}
\end{enumerate}
Finally, $f_G$ is a $(p^k, \frac{2p^{k-1}-p^{k-2}+1}{2}, \{1, p, p^k-p^{k-1}+1\})$ ZDF.
\end{proof}

\subsection{Others}
In this subsection, we propose more ZDFs by choosing different moduli $n$ and groups $G=\langle e \rangle$, according to Theorem \ref{th:construct_gzdb_on_generic_ring}.
\begin{thms}\label{th:construct_gzdb_on_primek2}
Suppose $s\ge 1$ is an integer. Let $n=p^k$, where $p$ is a prime number and $k \ge 2s$. Let $G=\langle p^{k-s}+1  \rangle$ be a subgroup of $\Z_n^\times$. Then the coset index function $f_G$ is a $(p^k, (sp+p-s)p^{k-s-1}, \bigcup_{i=0}^{s-1}{\{0, \sum_{j=0}^{i}{\varphi(p^{k-j})}\}})$ ZDF, where $\varphi$ is the Euler function.
\end{thms}
\begin{proof}The proof is similar with that of Theorem \ref{th:construct_gzdb_on_prime2}. We have
\begin{enumerate}[(1)]
\item $G= \{ 1 + tp^{k-s} \pmod{p^k}\mid t=0,1,\ldots,(p^s-1)\}$ and $|G|=p^s$;
\item for $\forall \alpha \in \Z_{n}$, it has
\begin{equation*}
|\alpha G| =\begin{cases}
1, & \alpha = 0,  \\
1, & p^i \mid\mid \alpha, s \le i \le k-1, \\
p^{s-i}, & p^i \mid\mid \alpha, 0 \le i \le s-1,
\end{cases}
\end{equation*}
where $p^i\mid\mid \alpha$ if $p^i$ is the largest power of $p$ dividing $\alpha$.
\item $|\Ima(f_G)|=\sum_{i=0}^{s-1}{\frac{\varphi(p^{k-i})}{p^{s-i}}}+\sum_{i=s}^{k}{\varphi(p^{k-i})}+1$=$(sp+p-s)p^{k-s-1}$;
\item for $\forall \alpha \in \Z_{p^k}\backslash\{0\}$, it has
\begin{equation*}\begin{split}
N(G, \alpha)=&\begin{cases}
0, & p^i \mid\mid \alpha, 1 \le i \le k-s-1, \\
\sum_{j=0}^{i}{\varphi(p^{k-j})}, & p^i \mid\mid \alpha, k-s \le i \le k-1,
\end{cases}
\end{split}\end{equation*}
\end{enumerate}
Finally, $f_G$ is a $(p^k, (sp+p-s)p^{k-s-1}, \bigcup_{i=0}^{s-1}{\{0, \sum_{j=0}^{i}{\varphi(p^{k-j})}\}})$ ZDF.
\end{proof}

In Theorem \ref{th:construct_gzdb_on_primek2}, let $s=1$, we obtain
\begin{prop}
Let $n=p^k$, where $p$ is a prime number and $k \ge 2$. Let $G=\langle p^{k-1}+1  \rangle$ be a subgroup of $\Z_n^\times$. Then the coset index function $f_G$ is a $(p^k, 2p^{k-1}-p^{k-2}, \{0, p^{k}-p^{k-1}\})$ ZDF.
\end{prop}

Here are two classes of $(n,m,S)$ ZDFs, where $n$ is not a prime power.
\begin{thms}\label{th:construct_gzdb_on_mp}
Let $n=mp$, where $m$ is a positive integer and $p$ is a prime number such that $\gcd(m,p)=1$. Suppose $s,t$ are two integers such that $st=p-1$. Let $e$ be an integer determined by the following system of equations:
\begin{equation*}
\begin{cases}
e \equiv 1  \pmod{m}  \\
e \equiv g^t \pmod{p},  \\
\end{cases}
\end{equation*}
where $g$ is a generator of $\Z_p^\times$. Let $G=\langle e \rangle$ be a subgroup of $\Z_n^\times$. Then the coset index function $f_G$ is an $(mp, m(1+t), \{0, m(s-1)\})$ ZDF.
\end{thms}
\begin{proof}The proof is similar with that of Theorem \ref{th:construct_gzdb_on_prime2}. We have
\begin{enumerate}[(1)]
\item $G= \{ e^i \pmod{mp}\mid i=0,1,\ldots,(s-1)\}$ and $|G|=s$;
\item for $\forall \alpha \in \Z_{n}$, it has
\begin{equation*}
|\alpha G| =\begin{cases}
1, & p \mid \alpha, \\
s, & p \nmid \alpha.
\end{cases}
\end{equation*}
\item $|\Ima(f_G)|=\frac{(p-1)m}{s}+m=m(1+t)$;
\item for $\forall \alpha \in \Z_{n}\backslash\{0\}$, it has
\begin{equation*}\begin{split}
N(G, \alpha)=&\begin{cases}
0, & p\mid \alpha, \\
m(s-1), & p \nmid \alpha.
\end{cases}
\end{split}\end{equation*}
\end{enumerate}
Finally, $f_G$ is an $(mp, m(1+t), \{0, m(s-1)\})$ ZDF.
\end{proof}

\begin{thms}\label{th:construct_gzdb_on_pp}
Let $n=p_1p_2$, where $p_1, p_2$ are two distinct prime numbers. Suppose $s_i,t_i$ are two integers such that $s_it_i=p_i-1$ ($i=1,2$). Denote $d=\gcd(s_1,s_2)$. Let $e$ be an integer determined by the following system of equations:
\begin{equation*}
\begin{cases}
e \equiv g_1^{t_1}  \pmod{p_1}  \\
e \equiv g_2^{t_2}  \pmod{p_2}  \\
\end{cases}
\end{equation*}
where $g_i$ is a generator of $\Z_{p_i}^\times$ ($i=1,2$). Let $G=\langle e \rangle$ be a subgroup of $\Z_n^\times$. Then the coset index function $f_G$ is a $(p_1p_2, 1+t_1+t_2+dt_1t_2, \{a_0,a_1,a_2\})$ ZDF, where $a_0=\frac{1}{d}(s_1s_2-s_1-s_2)+1$, $a_1=\frac{(p_1-1)s_2}{d}-p_1+s_2$, $a_2=\frac{(p_2-1)s_1}{d}-p_2+s_1$.
\end{thms}
\begin{proof}The proof is similar with that of Theorem \ref{th:construct_gzdb_on_prime2}. Let $r=\lcm(s_1,s_2)=\frac{s_1s_2}{\gcd(s_1,s_2)}=\frac{s_1s_2}{d}$. We have
\begin{enumerate}[(1)]
\item $G= \{ e^t \pmod{p_1p_2}\mid t=0,1,\ldots,(r-1)\}$ and $|G|=r$;
\item for $\forall \alpha \in \Z_{n}$, it has
\begin{equation*}
|\alpha G| =\begin{cases}
1, & \alpha = 0, \\
s_2, & p_1 \mid \alpha, \alpha \neq 0, \\
s_1, & p_2 \mid \alpha, \alpha \neq 0, \\
r, & \gcd(\alpha, n)=1.
\end{cases}
\end{equation*}
\item $|\Ima(f_G)|=1+\frac{p_2-1}{s_2}+\frac{p_1-1}{s_1}+\frac{(p_1-1)(p_2-1)}{r}=1+t_1+t_2+dt_1t_2$;
\item for $\forall \alpha \in \Z_{n}\backslash\{0\}$, it has
\begin{equation*}\begin{split}
N(G, \alpha)=&\begin{cases}
\frac{1}{d}(s_1s_2-s_1-s_2)+1, & \gcd(\alpha, n)=1, \\
\frac{(p_1-1)s_2}{d}-p_1+s_2, & p_1\mid \alpha, \\
\frac{(p_2-1)s_1}{d}-p_2+s_1, & p_2\mid \alpha.
\end{cases}
\end{split}\end{equation*}
\end{enumerate}
Finally, $f_G$ is a $(p_1p_2, 1+t_1+t_2+dt_1t_2, \{a_0,a_1,a_2\})$ ZDF, where $a_0=\frac{1}{d}(s_1s_2-s_1-s_2)+1$, $a_1=\frac{(p_1-1)s_2}{d}-p_1+s_2$, $a_2=\frac{(p_2-1)s_1}{d}-p_2+s_1$.
\end{proof}

\subsection{Summary}

In Table \ref{tb:gzdb-new}, we summarise all the ZDFs constructed by Theorem \ref{th:construct_gzdb_on_generic_ring} with subgroup $G=\langle e \rangle$ cyclic.

\begin{table}[!htbp]
\centering
\caption{\text{$(n,m,S)$ G-ZDB Functions Constructed in This Paper}}\label{tb:gzdb-new}
\begin{tabular*}{\linewidth }{p{0.08\linewidth}|p{0.18\linewidth}|p{0.23\linewidth}|p{0.24\linewidth}|p{0.15\linewidth}}
\hline
% after \\: \hline or \cline{col1-col2} \cline{col3-col4} ...
$n$ & $m$ & $S$ &  $e$  & Conditions \\
\hline
%\hline
$4$ & $3$ & $0, 2$  &  $3$ &  \\
\hline
$2^k$ & $2^{k-1}+1$ & $0, 2$   &  $2^{k-1}-1$ & $k>2$.  \\
\hline
$p^2$ & $p$ & $p$, $p^2-p+1$  &  $p-1$ & \tiny\makecell[tl]{$p$ is odd prime.}  \\
\hline
$p^k$ & $\frac{2p^{k-1}-p^{k-2}+1}{2}$ & $1$, $p$, $p^k-p^{k-1}+1$  &  $p^{k-1}-1$ & \tiny\makecell[tl]{$p$ is odd prime, \\ $k>2$.}  \\
\hline
$p^k$ & $2p^{k-1}-p^{k-2}$ & $0$, $p^{k}-p^{k-1}$  &  $p^{k-1}+1$ & \tiny\makecell[tl]{$p$ is prime, \\ $k\ge 2$.}  \\
\hline
$mp$ & $m(1+t)$ & $0$, $m(s-1)$  &
$\begin{aligned}\begin{cases}
e \equiv 1  \pmod{m}  \\
e \equiv g^t \pmod{p}
\end{cases}\end{aligned}$ & \tiny\makecell[tl]{$p$ is prime, \\ $\gcd(m,p)=1$, \\ $st=p-1$, \\ $\langle g \rangle = \Z_p^\times$.}  \\
\hline
$p_1p_2$ & \tiny$1+t_1+t_2+dt_1t_2$ & \makecell[tl]{\tiny$\frac{1}{d}(s_1s_2-s_1-s_2)+1$,  \\ \tiny$\frac{(p_1-1)s_2}{d}-p_1+s_2$, \\ \tiny$\frac{(p_2-1)s_1}{d}-p_2+s_1$} &
$\begin{aligned}\begin{cases}
e \equiv g_1^{t_1}  \pmod{p_1}  \\
e \equiv g_2^{t_2}  \pmod{p_2}
\end{cases}\end{aligned}$ & \tiny\makecell[tl]{$p_1\neq p_2$ are prime, \\ $s_it_i=p_i-1$, \\$\langle g_i \rangle = \Z_{p_i}^\times$($i=1,2$), \\ $d=\gcd(s_1,s_2)$.}  \\
\hline
\end{tabular*}
\end{table}

\section{Conclusion}\label{se:con}
In this paper, we have proposed a generic method to construct ZDF. Serval classes of  ZDFs with new parameters are constructed on $\Z_n$. Unfortunately, we find that these ZDFs cannot construct optimal objects. But these ZDFs may be valuable if more applications of ZDFs are discovered.

In the future work, we are expected to propose more ZDFs, by which we can construct optimal CWCs, DSSs and FHSs.

%\printbibliography
%\bibliographystyle{IEEEtranSN}
%\bibliography{GZDB}

% Generated by IEEEtranSN.bst, version: 1.14 (2015/08/26)

%\appendix
%\section{Examples}

\end{document}